\documentclass{amsart}
\usepackage{url}
\usepackage{graphicx}
\usepackage{orcidlink}

\newtheorem{theorem}{Theorem}[section]
\newtheorem{lemma}[theorem]{Lemma}

\theoremstyle{definition}
\newtheorem{definition}[theorem]{Definition}
\newtheorem{example}[theorem]{Example}

\theoremstyle{remark}

\numberwithin{equation}{section}

\begin{document}

\title{Introducing a Novel Subclass of Harmonic Functions with Close-to-Convex Properties}

\author{Serkan ÇAKMAK\textsuperscript{1} \orcidlink{0000-0003-0368-7672} \and Sibel YALÇIN\textsuperscript{2} \orcidlink{0000-0002-0243-8263}}

\address{\textsuperscript{1}Istanbul Gelisim University, Faculty of Economics, Administrative and Social Sciences, Department of Management Information Systems, Istanbul, Türkiye.}  
\email{secakmak@gelisim.edu.tr}  

\address{\textsuperscript{2}Bursa Uludag University, Faculty of Arts and Sciences, Department of Mathematics, Görükle, Bursa, Türkiye.}  
\email{syalcin@uludag.edu.tr}


\date{Manuscript received: August 21, 2024; accepted: January 18, 2025.}



\begin{abstract}
In this paper, we introduce a new subclass of close-to-convex harmonic functions. We present a sufficient coefficient condition for a function to be a member of this class. Furthermore, we establish a distortion theorem. These results lay the groundwork for extending the findings to function classes involving higher-order derivatives.
\end{abstract}

\maketitle

\section{Introduction}

In the realm of harmonic functions, every function $\mathfrak{f}$ belonging to the class $\mathcal{SH}^{0}$ can be expressed as $\mathfrak{f}=\mathfrak{u}+\overline{\mathfrak{v}}$, where

\begin{equation}
\mathfrak{u}(z)=z+\sum_{m=2}^{\infty}u_{m}z^{m},\quad \mathfrak{v}(z)=\sum_{m=2}^{\infty}v_{m}z^{m}
\label{eqH}
\end{equation}
with both functions being analytic in the open unit disk  $\mathbb{E}=\left\{ z\in \mathbb{C}:\left\vert z\right\vert <1\right\}$. Provided that $|\mathfrak{u}^{\prime}(z)| > |\mathfrak{v}^{\prime}(z)|$ in $\mathbb{E}$, $\mathfrak{f}$ is locally univalent and sense-preserving in $\mathbb{E}$. It's noteworthy that when $\mathfrak{v}(z)$ is identically zero, $\mathcal{SH}^{0}$ contracts to class $\mathcal{S}$.

The subclasses of $\mathcal{S}$ that map $\mathbb{E}$ onto starlike and close-to-convex domains, respectively, are denoted by $\mathcal{S^*}$ and $\mathcal{K}$. Similarly, $\mathcal{SH}^{0,*}$ and $\mathcal{KH}^{0}$ represent subclasses of $\mathcal{SH}^{0}$ that map $\mathbb{E}$ onto their respective domains. (For further details, refer to \cite{clunie1984harmonic, owa2002close})

In 2005, Gao and Zhou \cite{gao2005class} introduced the class

\begin{equation*}
\mathcal{K}_s=\left\{\mathfrak{\mathfrak{u}}\in \mathcal{S} : \mathrm{Re}\left\{\frac{z^2\mathfrak{u}^{\prime }(z)}{-\phi(z)\phi(-z)}\right\} >0 \text{ for }z\in \mathbb{E}\right\}.
\end{equation*}

In 2011, Şeker \cite{cseker2011certain} introduced the class  
\begin{equation*}  
\mathcal{K}(k, \gamma)=\left\{\mathfrak{u}\in \mathcal{S} : \mathrm{Re}\left\{\frac{z^k\mathfrak{u}^{\prime }(z)}{\phi_{k}(z)}\right\} > \gamma, \ \ \ 0 \leq \gamma < 1 \text{ and }z\in \mathbb{E}\right\}  
\end{equation*}  
where \( k \in \mathbb{Z}^{+} \), \( \phi(z) \) is an analytic function of the form  
\[
\phi(z)=z+\sum_{m=2}^\infty c_m z^m \in S^{*}\left (\frac{k-1}{k}\right),
\]
and \( \phi_k(z) \), which was previously introduced by Wang et al. \cite{Wang}, is defined as  
\begin{equation}  
    \phi_{k}(z)= \prod_{v=0}^{k-1} \mu^{-v} \phi(\mu^{v}z), \quad \text{where} \quad \mu^{k}=1 \label{gk}  
\end{equation}  
with \( \mu=e^{2\pi i /k} \).

The classes $\mathcal{K}_s$ and $\mathcal{K}(k, \gamma)$, introduced by Gao and Zhou \cite{gao2005class} and Şeker \cite{cseker2011certain}, respectively, focus solely on analytic functions and are independent of the variable $\overline{z}$. This restriction prevents the study of properties of harmonic functions that depend on the variable $\overline{z}$. This leads us to define a new function class that includes the co-analytic part of harmonic functions.

Observe that if \(k = 2\) in the class \(\mathcal{K}(k, \gamma)\), the class \(\mathcal{K}(\gamma)\) studied by Kowalczyk et al. \cite{kowalczyk2010subclass} is obtained. For \(k = 2\) and \(\gamma = 0\), the class \(\mathcal{K}_s\) studied by Gao and Zhou \cite{gao2005class} is obtained. It is clear that the class \(\mathcal{K}(k, \gamma)\) encompasses both classes. The class \(\mathcal{KH}^0(k, \gamma)\), which we will define shortly, covers the class \(\mathcal{K}(k, \gamma)\) since it is defined using harmonic conjugates of analytic functions belonging to the classes \(\mathcal{K}(k, \gamma)\). Therefore, the class \(\mathcal{KH}^0(k, \gamma)\) is a generalization of the \(\mathcal{K}(k, \gamma)\), \(\mathcal{K}_s\), and \(\mathcal{K}(\gamma)\) classes.

This generalization allows for a broader investigation of geometric properties, such as distortion bounds and close-to-convexity, for harmonic functions that depend on $\overline{z}$.

\begin{definition}
The class $\mathcal{KH}^{0}(k,\gamma)$ is defined as the collection of functions $\mathfrak{f}=\mathfrak{u}+\overline{\mathfrak{v}}\in \mathcal{SH}^{0}$ that adhere to the following inequality:
\begin{equation}
\mathrm{Re}\left\{\frac{z^k\mathfrak{u}^{\prime }(z)}{\phi_{k}(z)}-\gamma\right\}>\left\vert \frac{z^k\mathfrak{v}^{\prime }(z)}{\phi_{k}(z)}\right\vert,  \label{AH0}
\end{equation}
where $0 \leq \gamma < 1$ and $\phi_{k}(z)$ is given by \eqref{gk}. 
\end{definition} 

Specifically, when \(\mathfrak{v}(z) \equiv 0\), the class \(\mathcal{KH}^{0}(k, \gamma)\) reduces to the class \(\mathcal{K}(k, \gamma)\). Also, by setting \(\mathfrak{v}(z) \equiv 0\), \(k = 2\) and \(\gamma = 0\), we obtain \(\mathcal{KH}^{0}(2, 0) = \mathcal{K}_s\). The inclusion of the co-analytic term \(\overline{\mathfrak{v}(z)}\) extends these classes, providing a more general framework that accommodates both analytic and co-analytic components. Additionally, when \(\phi(z) = z\), and by appropriately selecting the parameters, the class \(\mathcal{KH}^0(k, \gamma)\) can be reduced to several well-known subclasses of harmonic functions, as outlined below.

\begin{itemize}
    \item [i] \(\mathcal{KH}^{0}(k, 0) = \mathcal{PH}^{0}\) \cite{Ponn2013}.
    \item [ii] \(\mathcal{KH}^{0}(k, \gamma) = \mathcal{PH}^{0}(\gamma)\) \cite{Lia, Lib}.
    \item [ii] \(\mathcal{KH}^{0}(k, 0) = \mathcal{WH}^{0}(0)\) (\cite{Ghosh}).
    \item [iv] \(\mathcal{KH}^{0}(k, \gamma) = \mathcal{WH}^{0}(0, \gamma)\) (\cite{Rajb}).
    \item [v] \(\mathcal{KH}^{0}(k, \gamma) = \mathcal{AH}^{0}(1, 0, \gamma)\) (\cite{Twms}).
    \item [vi] \(\mathcal{KH}^{0}(k, 0) = \mathcal{RH}^{0}(0, 0)\) (\cite{Yasar}).
\end{itemize}

For further details on harmonic function classes defined by differential inequality, see \cite{Bsh,Kalaj,Gh1,Gh2,Gh3,Firoz,SY,SY2,SC,EY,cakmak2022new}.

In this work, we investigate the distortion theorems and coefficient bounds for functions in the class $\mathcal{KH}^{0}(k,\gamma)$ and demonstrate that functions within this class exhibit close-to-convex behavior.

\section{Examples of Functions in the Class $\mathcal{KH}^{0}(k,\gamma)$}

\begin{example}\label{ex1}
    Let $\mathfrak{f} = \mathfrak{u} + \overline{\mathfrak{v}} = z + \frac{1-\gamma}{m}\overline{ z}^m$ and $\phi(z) = z$. For $0 \leq \gamma < 1$ and $\left \vert z \right \vert < 1$, we have
    \begin{align*} 
        \mathrm{Re} \left\{\frac{z^k\mathfrak{u}^{\prime}(z)}{\phi_k(z)} - \gamma\right\} &= 1 - \gamma>(1 - \gamma) \left \vert z \right \vert ^{m-1}=\left\vert \frac{z^k\mathfrak{v}^{\prime}(z)}{\phi_k(z)}\right\vert.
    \end{align*}
  Hence, $\mathfrak{f} \in \mathcal{KH}^{0}(k,\gamma)$.
\end{example}
The following examples can be given for the special case of the parameters in Example \ref{ex1}.

\begin{example}
    Let $\mathfrak{f}  = z +\frac{99}{200} \overline{ z}^2$, $\gamma=\frac{1}{100}$ and $\phi(z) = z$. Then $\mathfrak{f} \in \mathcal{KH}^{0}(k,\frac{33}{100}).$ The unit disk is mapped to a starlike region by the function $f$. The depiction in Figure \ref{fig1} showcases the image of the set $\mathbb{E}$ under the transformation defined by $\mathfrak{f}(z) = z + \frac{99}{200}\overline{z}^2$.
    
 \begin{figure}[h]
    \centering
    \includegraphics[width=1\textwidth]{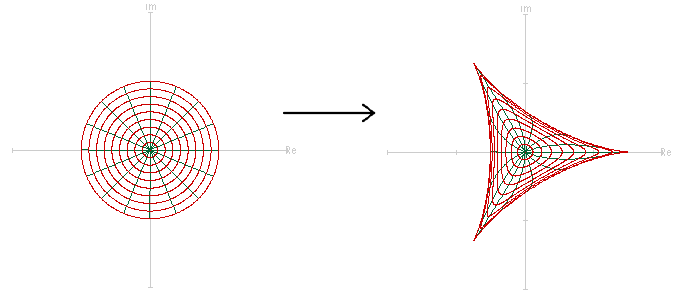} 
    \caption{Under the map $\mathfrak{f} = z + \frac{99}{200}\overline{z}^2$, the image of the unit disk.}
    \label{fig1}
\end{figure}
\end{example}

\begin{example}

Let $\mathfrak{f} = z +\frac{1}{10} \overline{ z}^2$, $\gamma=\frac{4}{5}$, and $\phi(z) = z$. Then $\mathfrak{f} \in \mathcal{KH}^{0}(k,\frac{4}{5})$. The unit disk is mapped to a convex region by the function $f$. The depiction in Figure \ref{fig3} showcases the image of the set $\mathbb{E}$ under the transformation defined by $\mathfrak{f} =  z +\frac{1}{10} \overline{ z}^2$.

    \begin{figure}[htbp]
    \centering
    \includegraphics[width=1\textwidth]{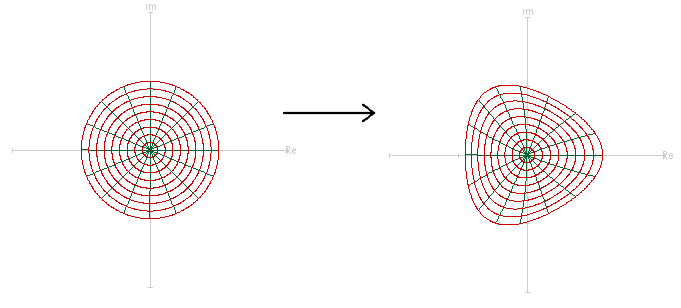}
    \caption{Under the map $\mathfrak{f} = z + \frac{1}{10}\overline{z}^2$, the image of the unit disk.}
    \label{fig3}
\end{figure}
\end{example}
\newpage
\begin{example}
    Let $\mathfrak{f}  = z + \frac{33}{100}\overline{ z}^3$, $\gamma=\frac{1}{100}$ and $\phi(z) = z$. Then $\mathfrak{f} \in \mathcal{KH}^{0}(k,\frac{33}{100}).$ The unit disk is mapped to a starlike region by the function $f$. The depiction in Figure \ref{fig4} showcases the image of the set $\mathbb{E}$ under the transformation defined by $\mathfrak{f} =  z +\frac{33}{100} \overline{ z}^3$.
    
    \begin{figure}[htbp]
    \centering
    \includegraphics[width=1\textwidth]{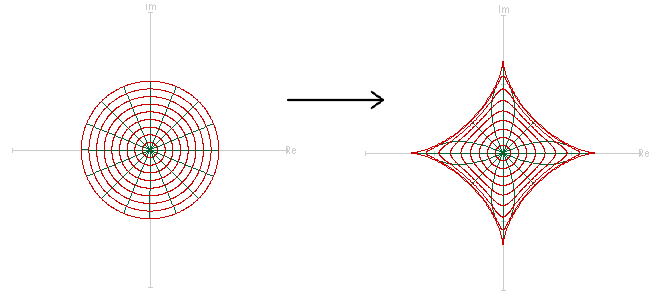}
    \caption{Under the map $\mathfrak{f} = z + \frac{33}{100}\overline{z}^3$, the image of the unit disk.}
    \label{fig4}
\end{figure}

\end{example}

\begin{example}
    Let $\mathfrak{f}  = z + \frac{1}{15}\overline{ z}^3$, $\gamma=\frac{4}{5}$ and $\phi(z) = z$. Then $\mathfrak{f} \in \mathcal{KH}^{0}(k,\frac{4}{5}).$ The unit disk is mapped to a convex region by the function $f$. The depiction in Figure \ref{fig6} showcases the image of the set $\mathbb{E}$ under the transformation defined by $\mathfrak{f} =  z +\frac{1}{15} \overline{ z}^3$.
    
    \begin{figure}[htbp]
    \centering
    \includegraphics[width=1\textwidth]{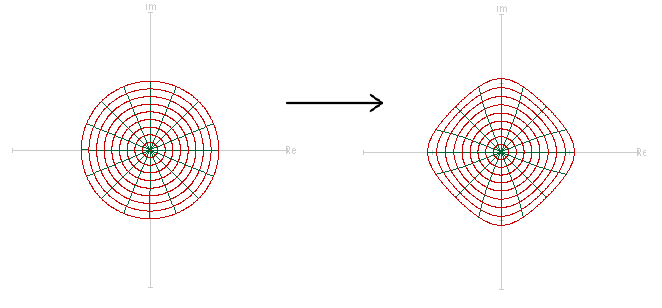}
    \caption{Under the map $\mathfrak{f} = z + \frac{1}{15}\overline{z}^3$, the image of the unit disk.}
    \label{fig6}
\end{figure}
\end{example}
\newpage
\begin{example}
    Let $\mathfrak{f}  = z +\frac{99}{500} \overline{ z}^5$, $\gamma=\frac{1}{100}$ and $\phi(z) = z$. Then $\mathfrak{f} \in \mathcal{KH}^{0}(k,\frac{99}{500}).$ The unit disk is mapped to a starlike region by the function $f$. The depiction in Figure \ref{fig7} showcases the image of the set $\mathbb{E}$ under the transformation defined by $\mathfrak{f} =  z +\frac{99}{500} \overline{ z}^5$.
    
    \begin{figure}[htbp]
    \centering
    \includegraphics[width=1\textwidth]{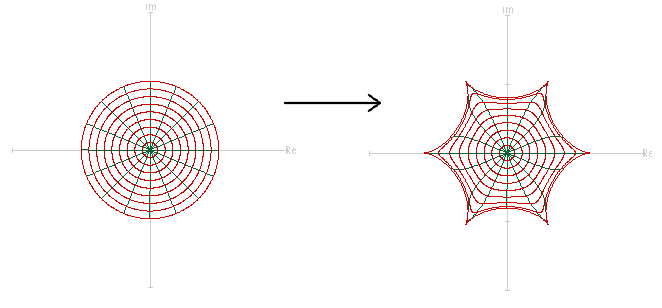}
    \caption{Under the map $\mathfrak{f} = z + \frac{99}{500}\overline{z}^5$, the image of the unit disk. }
    \label{fig7}
\end{figure}
\end{example}

\begin{example}
    Let $\mathfrak{f}  = z + \frac{1}{25}\overline{ z}^5$, $\gamma=\frac{4}{5}$ and $\phi(z) = z$. Then $\mathfrak{f} \in \mathcal{KH}^{0}(k,\frac{4}{5}).$ The unit disk is mapped to a convex region by the function $f$. 
    The depiction in Figure \ref{fig9} showcases the image of the set $\mathbb{E}$ under the transformation defined by $\mathfrak{f} =  z +\frac{1}{25} \overline{ z}^5$.
    
    \begin{figure}[htbp]
    \centering
    \includegraphics[width=1\textwidth]{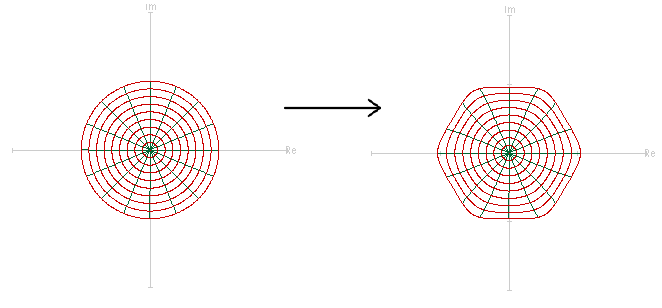}
    \caption{Under the map $\mathfrak{f} = z + \frac{1}{25}\overline{z}^5$, the image of the unit disk.}
    \label{fig9}
\end{figure}
\end{example}

\section{Geometric Properties of the class $\mathcal{KH}^{0}(k,\gamma)$ }

First, we give a result that establishes a sufficient condition for $\mathfrak{f}\in \mathcal{SH}^0$ to be close-to-convex, which comes from Clunie and Sheil-Small \cite{clunie1984harmonic}.

\begin{lemma}
\label{CTC lemma1} Let $\mathfrak{u}$ and $\mathfrak{v}$ be analytic functions in $\mathbb{E}$, such that $|\mathfrak{v}^{\prime}(0)| < |\mathfrak{u}^{\prime}(0)|$, and for each $\varepsilon$ $(|\varepsilon| = 1)$, $F_{\varepsilon }=\mathfrak{u}+\varepsilon \mathfrak{v}$ is close-to-convex. Then, $\mathfrak{f}=\mathfrak{u}+\overline{\mathfrak{v}}$ is close-to-convex in $\mathbb{E}$.
\end{lemma}

The result we will present now establishes a connection between the $\mathcal{KH}^{0}(k, \gamma)$ harmonic function class and the $\mathcal{K}(k, \gamma)$ analytic function class.

\begin{theorem}
\label{CTC theo1}  $\mathfrak{f}=\mathfrak{u}+\overline{\mathfrak{v}}\in \mathcal{KH}^{0}(k,\gamma)$ if and only if $F_{\varepsilon }=\mathfrak{u}+\varepsilon \mathfrak{v}\in \mathcal{K}(k,\gamma)$ for each $\varepsilon \ (\left\vert \varepsilon \right\vert =1)$.
\end{theorem}

\begin{proof}
Assume $\mathfrak{f}=\mathfrak{u}+\overline{\mathfrak{v}}\in \mathcal{KH}^{0}(k,\gamma)$. For each $\varepsilon \ (\left\vert \varepsilon \right\vert =1)$, we have
\begin{align*}
& {\text{Re}}\left \{\frac{z^k{F_{\varepsilon}}^{\prime }(z)}{\phi_{k}(z)}\right \} \\
& ={\text{Re}}\left \{\frac{z^k{\mathfrak{u}}^{\prime }(z)}{\phi_{k}(z)}\right \}+\varepsilon {\text{Re}}\left \{\frac{z^k{\mathfrak{v}}^{\prime }(z)}{\phi_{k}(z)}\right \} \\
& >{\text{Re}}\left \{\frac{z^k{\mathfrak{u}}^{\prime }(z)}{\phi_{k}(z)}\right \} -\left\vert \frac{z^k{\mathfrak{v}}^{\prime }(z)}{\phi_{k}(z)}\right\vert >\gamma \text{ }%
\left( z\in \mathbb{E}\right) .
\end{align*}
Hence, $F_{\varepsilon }\in \mathcal{K}(k,\gamma)$ for each $\varepsilon \ (\left\vert \varepsilon \right\vert =1)$.

Conversely, suppose $F_{\varepsilon }=\mathfrak{u}+\varepsilon \mathfrak{v}\in \mathcal{K}(k,\gamma)$. Then,
\begin{equation*}
{\text{Re}}\left \{\frac{z^k{\mathfrak{u}}^{\prime }(z)}{\phi_{k}(z)}\right \} >{\text{Re}}\left[ -\varepsilon \frac{z^k{\mathfrak{v}}^{\prime }(z)}{\phi_{k}(z)} %
\right] +\gamma \text{ }\left( z\in \mathbb{E}\right) .
\end{equation*}
Choosing an appropriate $\varepsilon \ (\left\vert \varepsilon \right\vert =1)$ yields
\begin{equation*}
{\text{Re}}\left \{\frac{z^k{\mathfrak{u}}^{\prime }(z)}{\phi_{k}(z)}-\gamma \right \} >\left\vert \frac{z^k{\mathfrak{v}}^{\prime }(z)}{\phi_{k}(z)}\right\vert \text{ }\left( z\in 
\mathbb{E}\right) ,
\end{equation*}
and thus $\mathfrak{f}\in \mathcal{KH}^{0}(k,\gamma)$.
\end{proof}

We will now demonstrate that harmonic functions belonging to the class $\mathcal{KH}^{0}(k,\gamma)$ map the open unit disk onto a close-to-convex region. To this end, we first state the following lemma and subsequently establish that functions in the class $\mathcal{K}(k,\gamma)$ are close-to-convex within the open unit disk.

\begin{lemma} \cite{Wang}
    Let $\phi(z)=z+\sum_{m=2}^\infty c_m z^m \in \mathcal{S^*}\left ( \frac{k-1}{k} \right )$ for $k \geq 1$. Then,
    \begin{equation}
        \Phi_{k}(z)=\frac{\phi_{k}(z)}{z^{k-1}}=z+\sum_{m=2}^\infty C_m z^m \in S^{*} \label{Gk}
    \end{equation}
    where $\phi_{k}(z)$ is given by \eqref{gk}.
\end{lemma}

\begin{theorem} \label{theo2}
    If $F$ is a function in the class $\mathcal{K}(k,\gamma)$, then $F$ is close-to-convex of order $\gamma$ in the region $\mathbb{E}$.
\end{theorem}

\begin{proof}
Let $F\in \mathcal{K}(k,\gamma)$. We have
\begin{equation*}
    {\text{Re}}\left \{\frac{z^k{F}^{\prime }(z)}{\phi_{k}(z)} \right \}={\text{Re}}\left \{\frac{z{F}^{\prime }(z)}{\Phi_{k}(z)} \right \}>\gamma,
\end{equation*}
where $\Phi_{k}(z)$ is given by \eqref{Gk}. Therefore, the function $F$ is close-to-convex of order $\gamma$ since $\Phi_{k}(z) \in \mathcal{S^*}$.
\end{proof}

\begin{theorem}
    Every function in the class $\mathcal{KH}^{0}(k,\gamma)$ is close-to-convex within the region $\mathbb{E}$.
\end{theorem}

\begin{proof}
Let $\mathfrak{f}=\mathfrak{u}+\overline{\mathfrak{v}}$ belong to class $\mathcal{KH}^{0}(k,\gamma)$. Then by Theorem \ref{CTC theo1} the function $F_{\varepsilon }=\mathfrak{u}+\varepsilon \mathfrak{v}$ belongs to class $\mathcal{K}(k,\gamma)$  and by Theorem \ref{theo2} also close to convex in $\mathbb{E}$. Therefore, by Lemma \ref{CTC lemma1}, $\mathfrak{f}=\mathfrak{u}+\overline{\mathfrak{v}}\in \mathcal{KH}^{0}(k,\gamma)$ is also close to convex in $\mathbb{E}$.
\end{proof}

In the following result, we derive a coefficient bound for functions belonging to the class \(\mathcal{KH}^0(k,\gamma)\).

\begin{theorem}\label{thecoef}
Let $\mathfrak{f}=\mathfrak{u}+\overline{\mathfrak{v}}\in \mathcal{KH}^{0}(k,\gamma)$. For $m\geq 2$, the following inequalities hold:
\begin{align*}
 & \ \left\vert u_{m}\right\vert +\left\vert v_{m}\right\vert \leq \gamma+m(1-\gamma).
\end{align*}%
For the function $\mathfrak{f}(z)=z+[\gamma+m(1-\gamma)] z^{m}$, every outcome is sharp and every equality is holds.
\end{theorem}

\begin{proof}
Suppose that $\mathfrak{f}=\mathfrak{u}+\overline{\mathfrak{v}}\in 
\mathcal{KH}^{0}(k,\gamma)$. $ F_{\varepsilon}=\mathfrak{u}+\varepsilon \mathfrak{v}\in \mathcal{K}(k,\gamma)\ $for $\varepsilon $ $\left( \left\vert \varepsilon \right\vert =1\right) $, according to Theorem \ref{CTC theo1}. With respect to every $\varepsilon \ (\left\vert \varepsilon \right\vert =1),$ we possess
 
\begin{equation*}
{\text{Re}}\left \{\frac{z^k{F_{\varepsilon }}^{\prime }(z)}{\phi_{k}(z)} \right \}={\text{Re}}\left \{\frac{z{F_{\varepsilon }}^{\prime }(z)}{\Phi_{k}(z)} \right \}={\text{Re}}\left \{\frac{z\left({\mathfrak{u}}^{\prime }(z)+\varepsilon \mathfrak{v}^{\prime }(z)\right)}{\Phi_{k}(z)} \right \}>\gamma.
\end{equation*}%
for $z \in \mathbb{E}$. On the other hand, there is an analytic function $\mathfrak{P}(z)=1+\underset{m=1}{\overset{\infty }{\sum }}\mathfrak{p}_{m}z^{m}$ in $\mathbb{E}$ whose real part is positive, satisfying
\begin{equation}
\frac{z\left({\mathfrak{u}}^{\prime }(z)+\varepsilon \mathfrak{v}^{\prime }(z)\right)}{\Phi_{k}(z)}=\gamma +\left( 1-\gamma  \right) \mathfrak{P}(z).
\end{equation}%
or
\begin{equation}
{z\left({\mathfrak{u}}^{\prime }(z)+\varepsilon \mathfrak{v}^{\prime }(z)\right)}=\left[\gamma +\left( 1-\gamma  \right) \mathfrak{P}(z)\right]\Phi_{k}(z).
\label{eqp}
\end{equation}%
Upon comparing the coefficients in (\ref{eqp}), it can be observed that
\begin{equation}
m (u_{m}+\varepsilon v_{m})=C_m+(1-\gamma)\mathfrak{p}_{m-1}+(1-\gamma)\mathfrak{p}_{1}C_{m-1}+\dots+(1-\gamma)\mathfrak{p}_{m-2}C_{2}. \label{eqpp}
\end{equation}%
Since $\Phi_{k}(z)$ is starlike, we have $\left \vert C_{m}  \right \vert\leq m$, and since $Re\{\mathfrak{P}(z)\}>0$, we have $\left\vert \mathfrak{p}_{m}\right\vert \leq 2$ for $m\geq 1$. Hence, by equation \ref{eqpp}, we have
\begin{equation}
m \left \vert u_{m}+\varepsilon v_{m}\right \vert \leq m\left [\gamma+m(1-\gamma)\right ].
\end{equation}
Since $\varepsilon
\left( \left\vert \varepsilon \right\vert =1\right) $ is arbitrary, it follows that the proof is concluded. The function $%
\mathfrak{f}(z)=z+$ $[\gamma+m(1-\gamma)] z^{m}$,
demonstrates the sharpness of inequality.
\end{proof}

Now, we provide the necessary coefficient condition for a harmonic function to belong to the class $\mathcal{KH}^{0}(k,\gamma)$.

\begin{theorem}
Let $\mathfrak{f}=\mathfrak{u}+\overline{\mathfrak{v}}\in \mathcal{SH}^{0}$ with the series expansions given by (\ref{eqH}). If the following inequality holds:
\begin{equation}
\sum\limits_{m=2}^{\infty }2m\left(
\left\vert u_{m}\right\vert +\left\vert v_{m}\right\vert \right)+\sum\limits_{m=2}^{\infty }\left( \left \vert 1-2\gamma \right \vert +1\right)
\left\vert C_{m}\right\vert  \leq 2(1-\gamma),  \label{coeffcond}
\end{equation}
then $\mathfrak{f}\in \mathcal{KH}^{0}(k,\gamma)$.
\end{theorem}

\begin{proof}
Consider $\mathfrak{u}$ and $\mathfrak{v}$ as functions with series expansions given by (\ref{eqH}). Let $F_{\varepsilon }=\mathfrak{u}+\varepsilon \mathfrak{v}$ for each $\varepsilon \ (\left \vert \varepsilon \right \vert =1)$. Define the functions $\phi_{k}(z)$ and $\Phi_{k}(z)$ as given by (\ref{gk}) and (\ref{Gk}), respectively. Then we can express $A$ as

\begin{eqnarray*}
  A &=&  \left \vert zF_{\varepsilon }^{\prime}-\frac{\phi_{k}(z)}{z^{k-1}} \right \vert - \left \vert zF_{\varepsilon }^{\prime}+\frac{(1-2\gamma) \phi_{k}(z)}{z^{k-1}} \right \vert \\
  &=& \left \vert z(\mathfrak{u}+\varepsilon \mathfrak{v})^{\prime}-\Phi_{k}(z) \right \vert - \left \vert z(\mathfrak{u}+\varepsilon \mathfrak{v})^{\prime}+(1-2\gamma) \Phi_{k}(z) \right \vert \\
  &=& \left \vert \sum_{m=2}^{\infty} m (u_{m} +\varepsilon v_{m}) z^m - \sum_{m=2}^{\infty} C_m z^m \right \vert \\ 
  &-& \left \vert (2-2\gamma)z+ \sum_{m=2}^{\infty} m (u_{m} +\varepsilon v_{m}) z^m + (1-2\gamma) \sum_{m=2}^{\infty} C_m z^m \right \vert\\
  &\leq & \sum_{m=2}^{\infty} m \left \vert u_{m} +\varepsilon v_{m}\right \vert \left \vert z \right \vert^m + \sum_{m=2}^{\infty} \left \vert C_m \right \vert \left \vert z \right \vert^m \\
  &-& \left ( (2-2\gamma)\left \vert z \right \vert - \sum_{m=2}^{\infty} m \left \vert u_{m} +\varepsilon v_{m}\right \vert \left \vert z \right \vert^m - \left \vert 1-2\gamma \right \vert\sum_{m=2}^{\infty} \left \vert C_m \right \vert \left \vert z \right \vert^m\right)\\
  & < & \left \{ -2(1-\gamma) +\sum_{m=2}^{\infty} 2m \left \vert u_{m} +\varepsilon v_{m}\right \vert +\left(\left \vert 1-2\gamma \right \vert +1 \right) \sum_{m=2}^{\infty} \left \vert C_m \right \vert \right \} \left \vert z \right \vert
\end{eqnarray*}

Since $\varepsilon
\left( \left\vert \varepsilon \right\vert =1\right) $ is arbitrary, and from the inequality (\ref{coeffcond}), we obtain that $A<0$. This implies that $F_{\varepsilon }=\mathfrak{u}+ \varepsilon \mathfrak{v}$ belongs to the class $\mathcal{K}(k,\gamma)$, and consequently, according to Theorem 2, it shows that $\mathfrak{f}=\mathfrak{u}+\overline{\mathfrak{v}}$ belongs to the class $\mathcal{KH}^{0}(k,\gamma)$.
\end{proof}

The result we present now provides the distortion bounds for functions in the class $\mathcal{KH}^{0}(k,\gamma)$.

\begin{theorem}
Assuming $\mathfrak{f}=\mathfrak{u}+\overline{\mathfrak{v}}\in \mathcal{KH}^{0}(k,\gamma)$, the following inequalities hold for all $z$:
\begin{equation*}
    \left\vert z\right\vert +\sum\limits_{m=2}^{\infty } (-1)^{m-1} [m(1-\gamma)+\gamma] \left \vert z  \right \vert ^{m} \leq \left\vert \mathfrak{f}(z)\right\vert \leq \left\vert z\right\vert +\sum\limits_{m=2}^{\infty } [m(1-\gamma)+\gamma] \left \vert z  \right \vert ^{m}.
\end{equation*}
These are sharp inequality for the function $\mathfrak{f}(z)=z+\sum\limits_{m=2}^{\infty }[m(1-\gamma)+\gamma]z^{m}$.
\end{theorem}

\begin{proof}
Let $\mathfrak{f}=\mathfrak{u}+\overline{\mathfrak{v}}\in \mathcal{KH}^{0}(k,\gamma).$ Then using Theorem \ref{CTC theo1}\textbf{,} $
F_{\varepsilon }=\mathfrak{u}+ \varepsilon \mathfrak{v}\in \mathcal{K}(k,\gamma)\ $for each $\varepsilon \ ( \left\vert \varepsilon \right\vert =1).$ Additionally, from Theorem \ref{thecoef} in \cite{cseker2011certain}, we obtain
\begin{eqnarray}
    \frac{1-(1-2\gamma)\left \vert z \right \vert}{(1+\left \vert z \right \vert)^3} \leq \left \vert F_{\varepsilon }^{\prime }(z) \right \vert \leq \frac{1+(1-2\gamma)\left \vert z \right \vert}{(1-\left \vert z \right \vert)^3}
\end{eqnarray}
Since 
\begin{eqnarray*}
\left\vert F_{\varepsilon }^{\prime }(z)\right\vert  &=&\left\vert \mathfrak{u}^{\prime }(z)+\varepsilon \mathfrak{v}^{\prime }(z)\right\vert  \\
&\leq &1+ \sum\limits_{m=2}^{\infty } m [m(1-\gamma)+\gamma] \left \vert z \right \vert ^{m-1}
\end{eqnarray*}%
and%
\begin{eqnarray*}
\left\vert F_{\varepsilon }^{\prime }(z)\right\vert  &=&\left\vert \mathfrak{u}^{\prime }(z)+\varepsilon \mathfrak{v}^{\prime }(z)\right\vert  \\
&\geq & 1+ \sum\limits_{m=2}^{\infty } (-1)^{m-1} m [m(1-\gamma)+\gamma] \left \vert z \right \vert ^{m-1},
\end{eqnarray*}%
in particular, we get 
\begin{equation*}
\left\vert \mathfrak{u}^{\prime }(z)\right\vert +\left\vert \mathfrak{v}%
^{\prime }(z)\right\vert \leq 1+ \sum\limits_{m=2}^{\infty } m [m(1-\gamma)+\gamma] \left \vert z \right \vert ^{m-1}
\end{equation*}
and%
\begin{equation*}
\left\vert \mathfrak{u}^{\prime }(z)\right\vert -\left\vert \mathfrak{v}%
^{\prime }(z)\right\vert \geq 1+ \sum\limits_{m=2}^{\infty } (-1)^{m-1} m [m(1-\gamma)+\gamma] \left \vert z \right \vert ^{m-1}.
\end{equation*}
Assume $\Gamma$ is the radial segment extending from $0$ to $z$, so
\begin{eqnarray*}
\left\vert \mathfrak{f}(z)\right\vert  &=&\left\vert \int\limits_{\Gamma }%
\frac{\partial \mathfrak{f}}{\partial \mathfrak{z} }d\mathfrak{z} +\frac{\partial 
\mathfrak{f}}{\partial \bar{\mathfrak{z}}}d\bar{\mathfrak{z}}\right\vert \leq
\int\limits_{\Gamma }\left( \left\vert \mathfrak{u}^{\prime }(\mathfrak{z}
)\right\vert +\left\vert \mathfrak{v}^{\prime }(\mathfrak{z} )\right\vert \right)
\left\vert d\mathfrak{z} \right\vert  \\
&\leq &\int\limits_{0}^{\left\vert z\right\vert }\left( 1+ \sum\limits_{m=2}^{\infty } m [m(1-\gamma)+\gamma] \left \vert \rho  \right \vert ^{m-1}\right) d\mathfrak{\rho } \\
&=&\left\vert z\right\vert +\sum\limits_{m=2}^{\infty }[m(1-\gamma)+\gamma] \left \vert z  \right \vert ^{m},
\end{eqnarray*}
and 
\begin{eqnarray*}
\left\vert \mathfrak{f}(z)\right\vert  &\geq &\int\limits_{\Gamma }\left(
\left\vert \mathfrak{u}^{\prime }(\mathfrak{z} )\right\vert -\left\vert \mathfrak{v}%
^{\prime }(\mathfrak{z} )\right\vert \right) \left\vert d\mathfrak{z} \right\vert  \\
&\geq &\int\limits_{0}^{\left\vert z\right\vert }\left( 1+ \sum\limits_{m=2}^{\infty } (-1)^{m-1} m [m(1-\gamma)+\gamma] \left \vert \rho  \right \vert ^{m-1}\right) d\mathfrak{\rho }
\\
&=&\left\vert z\right\vert +\sum\limits_{m=2}^{\infty } (-1)^{m-1} [m(1-\gamma)+\gamma] \left \vert z  \right \vert ^{m}.
\end{eqnarray*}
\end{proof}

\begin{theorem}
The class $\mathcal{KH}^{0}(k,\gamma)$ is closed under convex combinations.
\end{theorem}

\begin{proof}
Suppose $\mathfrak{f}_{\alpha}=\mathfrak{u}_{\alpha}+\overline{\mathfrak{v}_{\alpha}}\in 
\mathcal{KH}^{0}(k,\gamma)$ for $\alpha=1,2,...,p$ and $%
\sum\limits_{\alpha=1}^{p}s_{\alpha}=1$ $(0\leq s_{\alpha}\leq 1).$ The convex combination
of functions $\mathfrak{f}_{\alpha}$ $\left( \alpha=1,2,...,p\right) $ can be expressed as:
\begin{equation*}
\mathfrak{f}(z)=\sum\limits_{\alpha=1}^{p}s_{\alpha}\mathfrak{f}_{\alpha}(z)=\mathfrak{u}%
(z)+\overline{\mathfrak{v}(z)},
\end{equation*}%
where 
\begin{equation*}
\mathfrak{u}(z)=\sum\limits_{\alpha=1}^{p}s_{\alpha}\mathfrak{u}_{\alpha}(z)\text{ \ and \ }%
\mathfrak{v}(z)=\sum\limits_{\alpha=1}^{p}s_{\alpha}\mathfrak{v}_{\alpha}(z).
\end{equation*}%
Both $\mathfrak{u}$ and $\mathfrak{v}$ are analytic within the open unit disk $\mathbb{E}$, satisfying initial conditions $ \mathfrak{u}(0)=\mathfrak{v} (0)=\mathfrak{u}^{\prime }(0)-1=\mathfrak{v}^{\prime }(0)=0 $ and
\begin{eqnarray*}
{\text{Re}}\left[ \frac{z^{k}\mathfrak{u}^{\prime }(z)}{\phi_{k}(z)}-\gamma \right] &=&{\text{Re}}\left[ \sum%
\limits_{\alpha=1}^{p}s_{\alpha}\left( \frac{z^{k}\mathfrak{u}_{\alpha}^{\prime }(z)}{\phi_{k}(z)} - \gamma \right) \right] 
>\sum\limits_{\alpha=1}^{p}s_{\alpha}\left\vert \frac{z^{k}\mathfrak{v}_{\alpha}^{\prime }(z)}{\phi_{k}(z)} \right\vert 
= \left\vert \frac{z^{k}\mathfrak{v}^{\prime }(z)}{\phi_{k}(z)}\right\vert
\end{eqnarray*}%
showing that $\mathfrak{f}\in \mathcal{KH}^{0}(k,\gamma)$.
\end{proof}
\section{Conclusions}

In this paper, we introduced a new class of harmonic functions denoted by $\mathcal{KH}^{0}(k, \gamma)$. We established a relationship between $\mathcal{KH}^{0}(k, \gamma)$ and $\mathcal{K}(k, \gamma)$. We demonstrated that $\mathcal{KH}^{0}(k, \gamma)$ is close-to-convex. For functions in the $\mathcal{KH}^{0}(k, \gamma)$ class, we derived coefficient bounds and distortion theorems. Finally, we proved that $\mathcal{KH}^{0}(k, \gamma)$ is closed under convolution.

\bibliographystyle{amsplain}

\end{document}